\documentclass[absolute]{mymathart}
\usepackage[theorems]{mymathmacros}
\usepackage[usenames,dvipsnames]{color}
\usepackage{subfig}
\usepackage[shortlabels]{enumitem}
\usepackage{hyperref}
\usepackage{amssymb}
\usepackage{graphicx}
\usepackage{epigraph}
\usepackage{amsfonts,amssymb,color}
\usepackage[mathscr]{eucal}
\usepackage{amsmath, amsthm, amscd}
\usepackage{mathrsfs}
\usepackage{amsbsy}
\usepackage{float}
\usepackage{verbatim}
\usepackage{amsfonts} 
\usepackage{pifont}

\title[Continuity of measure-dimension mappings]{Continuity of measure-dimension mappings}

\author{Liangang Ma}

\address{Dept.\ of Mathematical Sciences, Binzhou University, Huanghe 5th Road No. 391, Binzhou 256600, Shandong, P. R. China} 
\email{maliangang000@163.com}
\thanks{The work is supported by ZR2019QA003 from SPNSF and 12001056 from NSFC}    

\newtheorem{theorem}[subsection]{Theorem}
\newtheorem{lemma}[subsection]{Lemma}
\newtheorem{Young's Lemma}[subsection]{Young's Lemma}
\newtheorem{MU Theorem}[subsection]{Mauldin-Urba\'nski Theorem}
\newtheorem{SSU Theorem}[subsection]{Simon-Solomyak-Urba\'nski Theorem}
\newtheorem{Prohorov's Theorem}[subsection]{Prohorov's Theorem}
\newtheorem{proposition}[subsection]{Proposition}
\newtheorem{corollary}[subsection]{Corollary}
\newtheorem{definition}[subsection]{Definition}
\newtheorem{exm}[subsection]{Example}
\newtheorem{rem}[subsection]{Remark}

\numberwithin{equation}{section}

\DeclareMathOperator*{\esssup}{ess\,sup}

\begin{document} 

\begin{abstract}

   We study continuity and discontinuity of the upper and lower (modified) box-counting, Hausdorff, packing, (modified) correlation measure-dimension mappings under the weak, setwise and TV topology on the space of Borel measures respectively in this work. We give various examples to show that no continuity can be guaranteed under the weak, setwise or TV topology on the space of Borel measures for any of these measure-dimension mappings. However, in some particular circumstances or by assuming some restrictions on the measures, we do have some (semi-)continuity results for some of these measure-dimension mappings under the  setwise topology.  In the end we point out some applications of our continuity results on deciding the dimensions of measures on ambient space with some dynamical structures.

 \end{abstract}
 
 \maketitle

\section{Introduction}\label{sec1}

The dimension of a measure on a metrizable space indicates how well the mass is distributed on the space. Its definition usually depends on the dimensions of sets in the space.  P. Mattila, M. Mor\'an and J. M. Rey gave some criterions on evaluating the well-posedness of different concepts of dimensions of measures, see \cite[p220]{MMR}. In their work they treat the dimensions of measures as as a function (functional) from the space of Borel measures to $\mathbb{R}^+=[0,\infty)$: the measure-dimension mappings. From this point of view it is a natural question to ask the continuity of various  measure-dimension mappings, with suitable topology endowed on the space of Borel measures. It would be interesting if the measure-dimension mappings could admit some continuity under some topology in case one is trying to deal with the corresponding dimensions of Borel measures.     

Considering the topological description of the space of Borel measures, the vague, weak, setwise and TV topology are well-known among probabilists, with their fineness in increasing order. One can refer to \cite{Kal2} and \cite{Kle} for the vague and weak topology, while to \cite{FKZ1} and \cite{Las} for the setwise and TV topology on the space of Borel measures. These topologies have lots of applications in various circumstances, for example, see the applications of the vague or weak topology  in particle systems  by Cox-Klenke-Perkins and Kallenberg \cite{CKP, Kal1},  the applications of the weak, setwise or TV topology in Markov decision processes by Feinberg-Kasyanov-Zgurovsky, O. Hern\'andez-Lerma and J. Lasserre in \cite{FKZ2, HL}. The choice of topology is subtle in different circumstances, with better properties possible under finer topology at the cost of more difficulty in guaranteeing convergence of concerning sequences of measures under the finer topology.

Luckily, we find that lots of measure-dimension mappings admit some semi-continuity properties under the setwise topology on the space of Borel measures. Especially, the upper measure-dimension mapping induced from some dimensional mapping satisfying the \emph{countable stability} admits lower semi-continuity under the setwise topology, in words of Falconer \cite[p40]{Fal1}. Upper semi-continuity always holds for any induced measure-dimension mappings under the setwise topology. However, some measure-dimension mappings do not admit any semi-continuity under any of these topologies, for example, the (modified) correlation measure-dimension mapping. We will also give examples to show that none of our concerning measure-dimension mappings admits any continuity if the topology is relaxed from the setwise level to weak level. There are some applications of our continuity results to Hausdorff dimensions of ergodic measures on some Riemannian manifold following Young \cite{You} at the end of the work.

\section{Notations, definitions and the main results}\label{sec2}

In this section we first introduce some popular measure-dimension mappings, then we give definitions of the vague, weak, setwise and TV topology on the space of Borel measures. At the end of the section we present our results on semi-continuity of some measure-dimension mappings under the setwise topology.

Let $(X,\rho)$ be a metric space with its Borel $\sigma$-algebra $\mathcal{B}$. Denote by $\mathcal{M}(X)$ to be the collection of all the Borel measures (note that it is possible that some measures are infinite in $\mathcal{M}(X)$). The collection of all the finite Borel measures on $(X,\mathcal{B})$ is denoted by $\hat{\mathcal{M}}(X)$. A \emph{dimensional mapping} is a real non-negative function
\begin{center}
$dim: \mathcal{B}\rightarrow \mathbb{R}^+$.
\end{center}
Of course the concept is interesting if and only if the function satisfies some reasonable properties, for example, the ones in \cite[p40]{Fal1}. We are particularly interested in those dimensional mappings satisfying the \emph{countable stability}, that is, for any collection of countable sets $\{A_i\}_{i=1}^\infty$,
\begin{center}
$dim(\cup_{i=1}^\infty A_i)=\sup \{dim(A_i)\}_{i=1}^\infty$.
\end{center}

Notable dimensional mappings in this work are the upper and lower (modified) box-counting, Hausdorff and packing dimensional mappings, which are denoted by 
\begin{center}
$\overline{dim}_B, \underline{dim}_B (\overline{dim}_{MB}, \underline{dim}_{MB}), dim_H, dim_P$
\end{center}
 respectively, see \cite{Fal1}. Denote by $dim_B$ ($dim_{MB}$) in case $\overline{dim}_B=\underline{dim}_B$ ($\overline{dim}_{MB}=\underline{dim}_{MB}$).  A \emph{measure-dimension mapping} is a real non-negative function
\begin{center}
$dim: \mathcal{M}(X)\rightarrow \mathbb{R}^+$.
\end{center}  
We abuse notations as one can easily distinguish its meaning of dimensional mapping or measure-dimension mapping from the contexts. Refer to \cite[p220]{MMR} on natural properties of reasonable measure-dimension mappings. The following two dual measure-dimension mappings are highlighted in various circumstances.  

\begin{definition}
Let $dim$ be a dimensional mapping on $(X,\rho)$. The induced \emph{lower} and \emph{upper measure-dimension mappings} from $\mathcal{M}(X)$ to $[0,\infty)$ are defined respectively to be:

\begin{center}
$dim^L(\nu)=\inf\{dim(A): \nu(A)>0, A\in\mathcal{B}\},$  
\end{center}
and
\begin{center}
$dim^U(\nu)=\inf\{dim(A): \nu(X\setminus A)=0, A\in\mathcal{B}\}$
\end{center}
for a measure $\nu\in\mathcal{M}(X)$. 
\end{definition}

It is easy to see that 
\begin{center}
$dim^L(\nu)\leq  dim^U(\nu)$
\end{center}
for any $\nu\in \mathcal{M}(X)$. When taking the dimensional mappings to be 
\begin{center}
$dim_B, \overline{dim}_B, \underline{dim}_B (\overline{dim}_{MB}, \underline{dim}_{MB}), dim_H, dim_P$,
\end{center}
the corresponding lower and upper measure-dimension mappings are denoted respectively by 
\begin{center}
$dim_B^L, \overline{dim}_B^L, \underline{dim}_B^L (\overline{dim}_{MB}^L, \underline{dim}_{MB}^L), dim_H^L, dim_P^L$
\end{center}
and
\begin{center}
$dim_B^U, \overline{dim}_B^U, \underline{dim}_B^U (\overline{dim}_{MB}^U, \underline{dim}_{MB}^U), dim_H^U, dim_P^U$.
\end{center}

In case the lower and upper measure-dimension mappings coincide with each other at some measure $\nu\in \mathcal{M}(X)$, denote the values simply by
\begin{center}
$dim_B(\nu), \overline{dim}_B(\nu), \underline{dim}_B(\nu) \big(\overline{dim}_{MB}(\nu), \underline{dim}_{MB}(\nu)\big), dim_H(\nu), dim_P(\nu)$.
\end{center}
Following Mattila-Mor\'an-Rey, we also pay attention to the correlation measure-dimension mapping $\dim_C $ and its modified version $\dim_{MC}$. The correlation measure-dimension mapping is introduced by P. Grassberger and I. Procaccia \cite{GP1, GP2}. 

\begin{definition}\label{def5}
The \emph{correlation dimension} of $\nu$ (see \cite{Cut}) is defined to be
\begin{center}
$dim_C(\nu)=\lim_{r\rightarrow 0}\cfrac{\log\int\nu(B(x,r))d\nu}{\log r}$
\end{center}
for a measure $\nu\in\mathcal{M}(X)$, in case the above limit exists.
\end{definition}

The modified one $\dim_{MC}$ is given by Pesin \cite{Pes}.

\begin{definition}\label{def6}
The \emph{modified correlation dimension} of $\nu$ is defined to be
\begin{center}
$dim_{MC}(\nu)=\lim_{\delta\rightarrow 0}\sup_{\{A\in\mathcal{B}: \nu(A)\geq 1-\delta\}}\lim_{r\rightarrow 0}\cfrac{\log\int_A\nu(B(x,r))d\nu}{\log r}$
\end{center}
for a measure $\nu\in\mathcal{M}(X)$, in case the above limit exists. 
\end{definition}

\begin{rem}
Regardless of the existence of the above two limits, one can define the \emph{upper} and \emph{lower correlation dimension} as well as the \emph{upper} and \emph{lower modified correlation dimension} as \cite[(2),(4)]{MMR}. We adopt these definitions directly because all the measures in our concerns in this work have the same upper and lower (modified) correlation dimensions.
\end{rem} 

By \cite[Proposition 10.2]{Fal2} and \cite[Lemma 2.8]{MMR}, we have 
\begin{center}
$dim_H^L(\nu)\leq dim_C(\nu)\leq dim_{MC}(\nu)$
\end{center}
for any $\nu\in \mathcal{M}(X)$.

In order to discuss the continuity of the measure-dimension mappings, some topology on $\mathcal{M}(X)$ is necessary. The following concepts apply to $\mathcal{M}(X)$ with topological ambient space $X$.

\begin{definition}
The \emph{vague topology} is the topology with basis 
\begin{center}
$\{\varrho\in\mathcal{M}(X): |\int_X f(x) d\varrho-\int_X f(x) d\nu|<\epsilon \}$
\end{center}
on $\mathcal{M}(X)$ for any continuous function $f: X \rightarrow\mathbb{R}$ with compact support and any real $\epsilon>0$. 
\end{definition}

Refer to  \cite{Kal2, Kle}. As the coarsest topology, no continuity can be guaranteed for any measure-dimension mappings in our concerns under it. However, we still would like to pose it here in case of particular interests for some readers. A finer topology is the weak topology.

\begin{definition}
The \emph{weak topology} on $\mathcal{M}(X)$ is the topology with basis 
\begin{center}
$\{\varrho\in\mathcal{M}(X): |\int_X f(x) d\varrho-\int_X f(x) d\nu|<\epsilon \}$
\end{center}
for any bounded continuous function $f: X \rightarrow\mathbb{R}$ and any real $\epsilon>0$. 
\end{definition}
Refer to \cite{Bil1, Bil2, Kle, Mat}. Unfortunately, still no (semi-)continuity can be guaranteed for any measure-dimension mappings in our concerns under the weak topology in general, see Theorem \ref{thm18}.

\begin{definition}
The \emph{setwise topology} on $\mathcal{M}(X)$ is the topology with basis 
\begin{center}
$\{\varrho\in\mathcal{M}(X): |\int_X f(x) d\varrho-\int_X f(x) d\nu|<\epsilon \}$
\end{center}
for any bounded measurable function $f: X \rightarrow\mathbb{R}$ and any real $\epsilon>0$. 
\end{definition}

Refer to \cite{Doo, FKZ1, GR, HL, Las, LY}. After escalating to the setwise topology, we are excited to find that some measure-dimension mappings admit some semi-continuity under it. Denote by 
\begin{center}
$\nu_n\stackrel{v}{\rightarrow}\nu, \nu_n\stackrel{w}{\rightarrow}\nu, \nu_n\stackrel{s}{\rightarrow}\nu,\nu_n\stackrel{TV}{\rightarrow}\nu$ 
\end{center}
as $n\rightarrow\infty$ for sequences of measures in $\mathcal{M}(X)$ converging under the vague, weak, setwise or TV topology (to be defined later) respectively.

\begin{theorem}\label{thm1}
Let $dim$ be a dimensional mapping satisfying the countable stability property.  The induced upper measure-dimension mapping $dim^U$ is lower semi-continuous under the setwise topology on $\mathcal{M}(X)$, that is, if $\nu_n\stackrel{s}{\rightarrow}\nu$ in $\mathcal{M}(X)$ as $n\rightarrow\infty$, then 
\begin{equation}\label{eq2}
\liminf_{n\rightarrow\infty} dim^U (\nu_n)\geq dim^U (\nu).
\end{equation}
\end{theorem}

Alternatively, we have  upper semi-continuity for the lower measure-dimension mapping $dim^L$, which shows the two measure-dimension mappings $dim^U$ and $dim^L$  are dual to each other in some sense. 

\begin{theorem}\label{thm4}
Let $dim$ be a dimensional mapping. The induced lower measure-dimension mapping $dim^L$ is upper semi-continuous under the setwise topology on $\mathcal{M}(X)$, that is, if $\nu_n\stackrel{s}{\rightarrow}\nu$ in $\mathcal{M}(X)$ as $n\rightarrow\infty$, then 
\begin{equation}
\limsup_{n\rightarrow\infty} dim^L (\nu_n)\leq dim^L (\nu).
\end{equation}
\end{theorem}

As the finest topology in this work, the TV topology is induced from an explicit metric on  $\mathcal{M}(X)$.  

\begin{definition}
The \emph{TV (total-variation) metric}  is defined to be
\begin{center}
$\Vert \nu-\varrho\Vert_{TV}=\sup_{A\in\mathcal{B}}\{|\nu(A)-\varrho(A)|\}$
\end{center}   
for two measures $\nu, \varrho\in\mathcal{M}(X)$.
\end{definition}

Refer to \cite{Doo, FKZ1, HL, Las, PS}). Considering the two measure-dimension mappings $dim_C$ and $dim_{MC}$, they do not admit any semi-continuity under the finest topology on $\mathcal{M}(X)$ in general, see Theorem \ref{thm5}.

\section{Proofs of the semi-continuity results and examples on discontinuity of the measure-dimension mappings under various topology}\label{sec7}

First, to provide a general intuition on the continuity of the measure-dimension mappings in our consideration, we present the readers with the following result.
\begin{theorem}\label{thm18}
The measure-dimension mappings 
\begin{center}
$dim_B^L, \overline{dim}_B^L, \underline{dim}_B^L (\overline{dim}_{MB}^L, \underline{dim}_{MB}^L), dim_H^L, dim_P^L,$
\end{center}
\begin{center}
$dim_B^U, \overline{dim}_B^U, \underline{dim}_B^U (\overline{dim}_{MB}^U, \underline{dim}_{MB}^U), dim_H^U, dim_P^U$ 
\end{center}
and
\begin{center}
$ dim_C, dim_{MC}$
\end{center}
from $\mathcal{M}(X)$ to $[0,+\infty)$ are not continuous under the weak, setwise or TV topology in general.
\end{theorem}
\begin{proof}
See for instance our Example \ref{exm6} and \ref{exm5}.
\end{proof}

\begin{rem}
It would be interesting to ask whether some measure-dimension mappings can be defined to be equipped with continuity under some topology, possibly at the cost of losing some known properties of some popular measure-dimension mappings.
\end{rem}

We first give an example to show that for a weakly convergent sequence $\nu_n\stackrel{w}{\rightarrow}\nu$ as $n\rightarrow\infty$, any of the sequences 
\begin{equation}\label{eq3}
\begin{array}{cc}
\{dim_B^L(\nu_n)\}_{n=1}^\infty, \{\overline{dim}_B^L(\nu_n)\}_{n=1}^\infty, \{\underline{dim}_B^L(\nu_n)\}_{n=1}^\infty, \{\overline{dim}_{MB}^L(\nu_n)\}_{n=1}^\infty,\\ 
\{\underline{dim}_{MB}^L(\nu_n)\}_{n=1}^\infty, \{dim_H^L(\nu_n)\}_{n=1}^\infty, \{dim_P^L(\nu_n)\}_{n=1}^\infty
\end{array}
\end{equation}
\begin{equation}\label{eq4}
\begin{array}{cc}
\{dim_B^U(\nu_n)\}_{n=1}^\infty, \{\overline{dim}_B^U(\nu_n)\}_{n=1}^\infty, \{\underline{dim}_B^U(\nu_n)\}_{n=1}^\infty, \{\overline{dim}_{MB}^U(\nu_n)\}_{n=1}^\infty,\\ 
\{\underline{dim}_{MB}^U(\nu_n)\}_{n=1}^\infty, \{dim_H^U(\nu_n)\}_{n=1}^\infty, \{dim_P^U(\nu_n)\}_{n=1}^\infty
\end{array}
\end{equation}
and
\begin{equation}\label{eq5}
\{dim_C(\nu_n)\}_{n=1}^\infty, \{dim_{MC}(\nu_n)\}_{n=1}^\infty
\end{equation}
may not converge. In this and all the examples below, we set $X\subset \mathbb{R}^d$ endowed with the Euclidean metric $\rho_u$ for some positive integer $d$. Let $\mathfrak{L}^d$ be the $d$-dimensional Lebesgue measure on $\mathbb{R}^d$.

\begin{exm}\label{exm1} 
Define a sequence of probability measures $\{\nu_n\}_{n\in\mathbb{N}}$ on $[0,1]$ to be 
\begin{center}
$ \nu_n=\left\{
\begin{array}{ll}
n\mathfrak{L}^1|_{[0,\frac{1}{n}]} & \mbox{ if } n \mbox{ is } odd,\\
\delta_{\frac{1}{n}} & \mbox{ if } n \mbox{ is } even.\\
\end{array}
\right.$ 
\end{center} 
in which $\delta_{\frac{1}{n}}$ is the Dirac measure at the point $\frac{1}{n}$. Let $\nu=\delta_{0}$ be the Dirac measure at $0$.
\end{exm} 

Obviously we have $\nu_n\stackrel{w}{\rightarrow}\nu$ as $n\rightarrow\infty$ (but not setwisely). By  \cite[Proposition 2.1]{You} and \cite[Lemma 2.8]{MMR}, direct computations give that 
\begin{center}
$
\begin{array}{ll}
& dim_B(\nu_n)=dim_{MB}(\nu_n)=dim_H(\nu_n)=dim_P(\nu_n)=dim_C(\nu_n)=dim_{MC}(\nu_n)\\
=& \left\{
\begin{array}{ll}
1 & \mbox{ if } n \mbox{ is } odd,\\
0 & \mbox{ if } n \mbox{ is } even.\\
\end{array}
\right.
\end{array}
$ 
\end{center}
So any of the sequences in (\ref{eq3})(\ref{eq4})(\ref{eq5}) does not converge in Example \ref{exm1}.

We then give an example to show that in case of $\nu_n\stackrel{w}{\rightarrow}\nu$, even if all the sequences 
\begin{center}
$\{dim_B(\nu_n)\}_{n=1}^\infty, \{dim_{MB}(\nu_n)\}_{n=1}^\infty, \{dim_H(\nu_n)\}_{n=1}^\infty$, $\{dim_P(\nu_n)\}_{n=1}^\infty, \{dim_C(\nu_n)\}_{n=1}^\infty, \{dim_{MC}(\nu_n)\}_{n=1}^\infty$
\end{center}
converge, their limits may not equal 
\begin{center}
$dim_B(\nu), dim_{MB}(\nu), dim_H(\nu), dim_P(\nu),  dim_C(\nu)$ and $dim_{MC}(\nu)$
\end{center}
respectively. The example is  borrowed from \cite[Example 2.2]{Bil1} essentially.

\begin{exm}[Billingsley]\label{exm4}
Define a sequence of probability measures $\nu_n$ on $[0,1]$ to be  
\begin{center}
$ \nu_n=\frac{1}{n}\Sigma_{i=1}^n\delta_{\frac{i}{n}}$ 
\end{center}
for $n\in\mathbb{N}$.
\end{exm} 
Obviously $ \nu_n\stackrel{w}{\rightarrow}\mathfrak{L}^1|_{[0,1]}$ as $n\rightarrow\infty$ (but not setwisely). However, we have
\begin{center}
$\lim_{n\rightarrow\infty}dim_B(\nu_n)=\lim_{n\rightarrow\infty} dim_{MB}(\nu_n)=\lim_{n\rightarrow\infty}dim_H(\nu_n)=\lim_{n\rightarrow\infty}dim_P(\nu_n)=\lim_{n\rightarrow\infty} dim_C(\nu_n)=\lim_{n\rightarrow\infty} dim_{MC}(\nu_n)=0$,
\end{center}
while
\begin{center}
$dim_B(\mathfrak{L}^1|_{[0,1]})=dim_{MB}(\mathfrak{L}^1|_{[0,1]})=dim_H(\mathfrak{L}^1|_{[0,1]})=dim_P(\mathfrak{L}^1|_{[0,1]})=dim_C(\mathfrak{L}^1|_{[0,1]})=dim_{MC}(\mathfrak{L}^1|_{[0,1]})=1$.
\end{center}

One might think that things will be better if the atomic measures are excluded, however, the following example shows that the discontinuity still exists among sequences of non-atomic measures.

\begin{exm}\label{exm3}
Suppose $k\geq 2$ is an integer. Let
\begin{center}
$S=\{s_i: [0,1]\rightarrow[0,1]\}_{i=1}^k$ 
\end{center}
be an affine IFS with the contraction ratio $0<|s_i|:=|s_i'(x)|<1$ being a fixed number for any $1\leq i\leq k$ on $X=[0,1]$. We require it satisfies the \emph{strong separation condition}:
\begin{center}
$s_i([0,1])\cap s_j([0,1])=\emptyset$
\end{center}
for any $1\leq i\neq j\leq k$. We do not specify the terminals of the the intervals $\{s_i([0,1])\}_{i=1}^k$ as they have no effect on the properties we would like to demonstrate through the example. Let $0<h<1$ be the unique solution of the \emph{Bowen equation}
\begin{center}
$\sum_{i=1}^k|s_i|^h=1$.
\end{center}
There is an unique ergodic measure (with respect to the push-forward of the shift map under the projection) $\nu$ on the attractor $J$ in this case. 

Let 
\begin{center}
$\mathbb{N}_k:=\{1,2,\cdots,k\}$ 
\end{center}
 be the $k$-truncation of $\mathbb{N}$. For every fixed $n\in\mathbb{N}$, let
\begin{center}
$\mathbb{N}_k^n=\{\omega: \omega=\omega_1\omega_2\cdots\omega_n, \omega_i\in\mathbb{N}_k \mbox{\ for any } 1\leq i\leq n\}$
\end{center}     
be the collection of length-$n$ concatenation-words of $\mathbb{N}_k$. Let 
\begin{center}
$X_\omega=s_\omega([0,1])=s_{\omega_1}\circ s_{\omega_2}\cdots \circ s_{\omega_n}([0,1])$
\end{center}
for any $\omega\in\mathbb{N}_k^n$ and any $n\in\mathbb{N}$. Define a sequence of probability measures $\{\nu_n\}_{n\in\mathbb{N}}$ to be
\begin{center}
$\nu_n=\sum_{\omega\in\mathbb{N}_k^n}|s_\omega|^h\mathfrak{L}^1|_{X_\omega}=\sum_{\omega\in\mathbb{N}_k^n}(|s_{\omega_1}||s_{\omega_2}|\cdots|s_{\omega_n}|)^h\mathfrak{L}^1|_{X_\omega}$
\end{center} 
on $[0,1]$ for any $n\in\mathbb{N}$. It is supported on $\cup_{\omega\in\mathbb{N}_k^n} X_\omega$ for each $n\in\mathbb{N}$ obviously. In fact $\nu_n$ is the $n$-th canonical mass distribution of the linear cookie-cutter map in each cutting step (see for example, the \emph{self-similar} measures \cite{MR}). 
\end{exm}

Note that none of the measures above is atomic. It is easy to show that for any open set $L\subset [0,1]$, we have $ \nu_n(L)\geq \nu(L)$. So by  \cite[Theorem 2.1]{Bil1}, we have 
\begin{center}
$\nu_n\stackrel{w}{\rightarrow}\nu$ 
\end{center}
as $n\rightarrow\infty$ (but not setwisely). One can show that (left to the readers)  
\begin{center}
$\lim_{n\rightarrow\infty}dim_B(\nu_n)=\lim_{n\rightarrow\infty} dim_{MB}(\nu_n)=\lim_{n\rightarrow\infty}dim_H(\nu_n)=\lim_{n\rightarrow\infty}dim_P(\nu_n)=\lim_{n\rightarrow\infty} dim_C(\nu_n)=\lim_{n\rightarrow\infty} dim_{MC}(\nu_n)=1$,
\end{center}
while
\begin{center}
$dim_B(\nu)=dim_{MB}(\nu)=dim_H(\nu)=dim_P(\nu)=dim_C(\nu)=dim_{MC}(\nu)=h$.
\end{center}

One might also hope that escalating the strength of convergence of measures may save the continuity, but we will give an example to show that  convergence of dimensions is not true even for TV convergent sequences of measures (see also Example \ref{exm5}).

\begin{exm}\label{exm6}
Define a sequence of probability measures $\nu_n$ on $[0,2]$ to be 
\begin{center}
$ \nu_n=\frac{n-1}{n} \delta_{0}+\frac{1}{n} \mathfrak{L}^1|_{[1,2]}.$ 
\end{center} 
\end{exm}
It is easy to see that $\| \nu_n-\delta_{0} \|_{TV}=\frac{1}{n}\rightarrow 0$ as $n\rightarrow\infty$, so $ \nu_n\stackrel{TV}{\rightarrow}\delta_{0}$. However,

\begin{center}
$
\begin{array}{ll}
&\lim_{n\rightarrow\infty} dim_B^U(\nu_n)=\lim_{n\rightarrow\infty} dim_{MB}^U(\nu_n)=\lim_{n\rightarrow\infty} dim_H^U(\nu_n)=\lim_{n\rightarrow\infty}dim_P^U(\nu_n)=1\\
>& 0=dim_B(\delta_{0})=dim_{MB}(\delta_{0})=dim_H(\delta_{0})=dim_P(\delta_{0}).
\end{array}
$
\end{center}

In fact, continuity of these measure-dimension mappings has no chance to be true if there is no dimensional restrictions on the sequences in (\ref{eq3})(\ref{eq4})(\ref{eq5}). Luckily we do have some semi-continuity property for some measure-dimension mappings under the setwise topology. Now we prove Theorem \ref{thm1}.\\

Proof of Theorem \ref{thm1}:

\begin{proof}
Assume that (\ref{eq2}) does not hold, that is,  $\liminf_{n\rightarrow\infty} dim^U(\nu_n)< dim^U(\nu)$. Then we can find a sequence of positive integers $\{n_k\}_{k=1}^\infty$ and a real number $a<dim^U(\nu)$, such that
\begin{center}
$\lim_{k\rightarrow\infty} dim^U(\nu_{n_k})=a$.
\end{center} 
So there exists a sequence of measurable sets $\{A_k\}_{k=1}^\infty$, such that 
\begin{center}
$\nu_{n_k}(X\setminus A_k)=0$ and $dim(A_k)<a+\varepsilon< dim^U(\nu)$
\end{center}
for some small $\varepsilon>0$. Now let $A=\cup_{k=1}^\infty A_k$. Since $ \nu_n\stackrel{s}{\rightarrow}\nu$, we have
\begin{center}
$\nu(X\setminus A)=\lim_{k\rightarrow\infty}\nu_{n_k}(X\setminus A)=0.$
\end{center}
However, as $dim$ satisfies the countable stability property, then
\begin{center}
$dim(A)=\sup_k\{dim(A_k)\}<a+\varepsilon<dim^U(\nu)$. 
\end{center}
This contradicts the definition of $dim^U(\nu)$, which justifies our theorem.
\end{proof}

\begin{rem}\label{rem1}
Theorem \ref{thm1} does not hold for weakly convergent sequences $\nu_n\stackrel{w}{\rightarrow}\nu$ as $n\rightarrow\infty$ in $\mathcal{M}(X)$, as one can see from our Example \ref{exm4}. Together with Example \ref{exm3}, one can see that both lower semi-continuity and upper semi-continuity are not true for general measure-dimension mappings under the weak topology. 
\end{rem}

Theorem \ref{thm1} has some interesting extensions in some special cases.

\begin{corollary}\label{cor1}
Let $dim$ be a dimensional mapping satisfying the countable stability property. Let $dim^U$ be the induced upper measure-dimension mapping. For $ \nu_n\stackrel{s}{\rightarrow}\nu$ as $n\rightarrow\infty$, if
\begin{center}
$\esssup\{dim^U(\nu_n)\}_{n=1}^\infty\leq dim^U(\nu)$,
\end{center}
 then 
\begin{center}
$\lim_{n\rightarrow\infty} dim^U(\nu_n)=dim^U(\nu)$.
\end{center}
\end{corollary}

\begin{proof}
This follows instantly from Theorem \ref{thm1}.
\end{proof}

\begin{corollary}\label{cor4}
The measure-dimension mappings $dim_{MB}^U, dim_H^U, dim_P^U$ are all lower semi-continuous under the setwise topology on $\mathcal{M}(X)$.
\end{corollary}

\begin{proof}
This is because they are respectively induced from dimensional mappings 
\begin{center}
$dim_{MB}, dim_H, dim_P$ 
\end{center}
with countable stability property, refer to \cite{Fal1}.
\end{proof}

As to $dim_B^U$ ($\overline{dim}_B^U, \underline{dim}_B^U$), since the dimensional mapping $dim_B$ ($\overline{dim}_B, \underline{dim}_B$) inducing it does not satisfy countable stability, Theorem \ref{thm1} does not apply to it. The following example shows that it is not lower semi-continuous under the setwise topology.

\begin{exm}\label{exm8}
Let $X=\{\frac{1}{i}\}_{i=1}^\infty$. Consider the sequence of measures 
\begin{center}
$ \nu_n=\sum_{i=1}^n\frac{1}{i^2} \delta_{\frac{1}{i}}$ 
\end{center} 
for $n\in\mathbb{N}$ on $X$. Let 
\begin{center}
$ \nu=\sum_{i=1}^\infty\frac{1}{i^2} \delta_{\frac{1}{i}}$.
\end{center} 
\end{exm}

Direct computations (left to the readers) show that $\nu_n\stackrel{TV}{\rightarrow}\nu$ while
\begin{center}
$\lim_{n\rightarrow\infty} dim_B^U(\nu_n)=0<\frac{1}{2}=dim_B^U(\nu)$
\end{center}
in Example \ref{exm8}. These measures can be normalized to be probabilities.

Note that the proof of Theorem \ref{thm1} can not be applied to the measure-dimension mapping $dim^L$, because if we choose a sequence of sets $\{A_k\}_{k=1}^\infty$, such that 
\begin{center}
$\nu_{n_k}(A_k)>0$ and $dim^L(A_k)<a+\varepsilon< dim^L(\nu)$
\end{center}
we can not guarantee $\nu(A)=\lim_{k\rightarrow\infty}\nu_{n_k}(A)>0.$ In fact one can see from the following example that lower semi-continuity of typical measure-dimension mapping $dim^L$ is usually not true under setwise topology on $\mathcal{M}(X)$.

\begin{exm}\label{exm5}
Define a sequence of probability measures $\{\nu_n\}_{n\in\mathbb{N}}$ on $[0,1]$ to be 
\begin{center}
$ \nu_n=\frac{1}{n} \delta_{0}+\mathfrak{L}^1|_{[\frac{1}{n},1]}$ 
\end{center} 
for $n\in\mathbb{N}$.
\end{exm}
It is easy to see that $\| \nu_n-\mathfrak{L}^1|_{[0,1]} \|_{TV}=\frac{1}{n}\rightarrow 0$ as $n\rightarrow\infty$, so $ \nu_n\stackrel{TV}{\rightarrow}\delta_{0}$. However,

\begin{center}
$\lim_{n\rightarrow\infty}dim_B^L(\nu_n)=\lim_{n\rightarrow\infty}dim_{MB}^L(\nu_n)=\lim_{n\rightarrow\infty}dim_H^L(\nu_n)=\lim_{n\rightarrow\infty}dim_P^L(\nu_n)=0<1=dim_B(\mathfrak{L}^1|_{[0,1]})=dim_{MB}(\mathfrak{L}^1|_{[0,1]})=dim_H(\mathfrak{L}^1|_{[0,1]})=dim_P( \mathfrak{L}^1|_{[0,1]})$.
\end{center}

Now we prove the semi-continuity result for the measure-dimension mapping $dim^L$.\\

Proof of Theorem \ref{thm4}:

\begin{proof}
Accoriding to the definition of  $dim^L(\nu)$, for any small $\varepsilon>0$, we can find a measurable set $A$  with $\nu(A)>0$ and  $dim(A)\leq dim^L(\nu)+\varepsilon$. Since $\nu_n\stackrel{s}{\rightarrow}\nu$  as $n\rightarrow\infty$, we can guarantee that 
\begin{center}
$\nu_n(A)>0$
\end{center}
for any $n$ large enough. This implies that 
\begin{center}
$\limsup_{n\rightarrow\infty} dim^L(\nu_n)\leq dim^L(\nu)+\varepsilon$.
\end{center}
The proof is finished by letting $\epsilon\rightarrow 0$.

\end{proof}

The following result is instant in virtue of Theorem \ref{thm4}.

\begin{corollary}\label{cor3}
The measure-dimension mappings $dim_B^L (\overline{dim}_B^L, \underline{dim}_B^L), dim_{MB}^L, dim_H^L, dim_P^L$ are all upper semi-continuous under the setwise topology on $\mathcal{M}(X)$.
\end{corollary}

Considering Corollary \ref{cor4} and \ref{cor3}, it would be interesting to ask whether  the two measure-dimension mappings $dim_C$ or $dim_{MC}$ admits any semi-continuity. It turns out that they  do not have any semi-continuity under the TV topology in general. 

\begin{theorem}\label{thm5}
The two measure-dimension mappings $dim_C$ and $dim_{MC}$ from $\mathcal{M}(X)$ to $[0,+\infty)$  are neither upper semi-continuous nor lower semi-continuous under the TV topology.  
\end{theorem}
\begin{proof}
Simple calculations show that, in Example \ref{exm5},
\begin{center}
$dim_C(\nu_n)=dim_{MC}(\nu_n)=\lim_{r\rightarrow 0}\cfrac{\log \int\nu_n(B(x,r))d\nu_n}{\log r}=\cfrac{\log\big((\frac{1}{n})^2+2r\frac{n-1}{n}\big)}{\log r}=0$,
\end{center} 
while
\begin{center}
$dim_C(\mathfrak{L}^1|_{[0,1]})=dim_{MC}(\mathfrak{L}^1|_{[0,1]})=1$.
\end{center}
Thus Example \ref{exm5} justifies the fact that the two measure-dimension mappings $dim_C$ and $dim_{MC}$ can not be lower semi-continuous under the TV topology. For an example violating the upper semi-continuity of $dim_C$ and $dim_{MC}$ under the TV topology, see Example \ref{exm7}.
\end{proof}

However, in virtue of \cite[Theorem 2.6]{MMR}, if some restrictions are set on the sequence of the measures $\{\nu_n\}_{n=1}^\infty$, we have some partial semi-continuity results.
\begin{corollary}
For a sequence of  measures $\{\nu_n\in\mathcal{M}(X)\}_{n=1}^\infty$, if  $\nu_n$ has an essentially bounded density with respect to $\nu$ (which implies that $\nu_n$ is absolutely continuous with respect to $\nu$) for any $n\in\mathbb{N}$, then 
\begin{center}
$\liminf_{n\rightarrow\infty} dim_C(\nu_n)\geq dim_C(\nu)$.
\end{center} 
\end{corollary}

The restriction on the sequence of the measures $\{\nu_n\}_{n=1}^\infty$ can be relaxed, by \cite[Theorem 2.10]{MMR}, to give the following lower semi-continuity result for $dim_{MC}$ in due course. 
\begin{corollary}
For any sequence of  measures $\{\nu_n\in\mathcal{M}(X)\}_{n=1}^\infty$, if  $\nu_n$ is absolutely continuous with respect to $\nu$ for any $n\in\mathbb{N}$, then 
\begin{center}
$\liminf_{n\rightarrow\infty} dim_{MC}(\nu_n)\geq dim_{MC}(\nu)$.
\end{center} 
\end{corollary}

Considering Theorem \ref{thm5}, we still owe the readers an example of a convergent sequence of measures violating the upper semi-continuity of $dim_C$ and $dim_{MC}$ under the TV topology. The following example is an evolution of \cite[Example 2.5]{MMR}.

\begin{exm}[Mattila-Mor\'an-Rey]\label{exm7}
Let $a\in(0,1)$ be a real number. Consider the sequence $\{a^{n^2}\}_{n=0}^\infty\subset [0,1]$. 
Define a sequence of probability measures $\nu_n$ on $[0,1]$ to be 
\begin{center}
$\nu_n=\cfrac{1-a}{1-a^{n+1}}\sum_{i=0}^n\cfrac{a^i}{a^{i^2}-a^{{(i+1)}^2}}\mathfrak{L}^1|_{[a^{i^2},a^{{(i+1)}^2}]}$
\end{center}
for any $n\in\mathbb{N}$. Let $\nu$ be the probability measure on $[0,1]$ with
\begin{center}
$\nu=(1-a)\sum_{i=0}^\infty \cfrac{a^i}{a^{i^2}-a^{{(i+1)}^2}}\mathfrak{L}^1|_{[a^{i^2},a^{{(i+1)}^2}]}$.
\end{center}

\end{exm}

First we show the sequence of measures converges under the TV topology.
\begin{lemma}\label{lem7}
For the sequence of measures $\{\nu_n\}_{n\in\mathbb{N}}$ in Example \ref{exm7}, we have
\begin{center}
$\nu_n\stackrel{TV}{\rightarrow}\nu$ as $n\rightarrow\infty$.
\end{center}
 \end{lemma}
 
\begin{proof}
In this case we have the following estimation on the total variation of the difference between the two measures $\nu_n$ and $\nu$,
\begin{center}
$\|\nu_n-\nu\|_{TV}=\max\{a^{n+1},\cfrac{a^{n+1}}{1-a^{n+1}}\}=\cfrac{a^{n+1}}{1-a^{n+1}}$
\end{center}
for any $n\in\mathbb{N}$. So $\lim_{n\rightarrow\infty}\|\nu_n-\nu\|_{TV}=0$, which justifies its convergence under the TV topology.
\end{proof} 

Although the correlation dimension and modified correlation dimension appear as global concepts, they are sometimes dominated by some properties on a subset of $X$ of relatively small measure, as indicated by the following result.

\begin{lemma}\label{lem6}
Let $\nu\in\mathcal{M}(X)$ be a probability measure. For any $r$ small enough, if there exists a subset $Y_r\subset X$ of positive measure, such that for \emph{a.e.} $x\in Y_r$, the following estimation holds,
\begin{equation}\label{eq28}
\nu(B(x,r))\nu(Y_r)\geq r^{o(1)},
\end{equation}
in which $o(1)$ is a positive term satisfying $\lim_{r\rightarrow 0} o(1)=0$. Then we have
\begin{center}
$dim_C(\nu)=dim_{MC}(\nu)=0$.
\end{center}
\end{lemma}
\begin{proof}
Under the above assumptions, we have
\begin{center}
$\int \nu(B(x,r))d\nu(x)\geq \nu(B(x,r))\nu(Y_r)$.
\end{center}
So 
\begin{center}
$\log \int \nu(B(x,r))d\nu(x)\geq\log \big(\nu(B(x,r))\nu(Y_r)\big)\geq o(1)\log r$,
\end{center}
which gives that
\begin{center}
$\cfrac{\log \int \nu(B(x,r))d\nu(x)}{\log r}\leq o(1)$.
\end{center}
The proof ends by letting $r\rightarrow 0$ in the above estimation.  
\end{proof}

Now we pay attention to the correlation and modified correlation dimensions of the limit measure of the sequence $\{\nu_n\}_{n\in\mathbb{N}}$ in Example \ref{exm7}.
\begin{proposition}[Mattila-Mor\'an-Rey]\label{pro1}
For the limit measure $\nu$ in Example \ref{exm7}, we have
\begin{center}
$dim_C(\nu)=dim_{MC}(\nu)=0$.
\end{center}
\end{proposition}
\begin{proof}
We achieve the conclusion by Lemma \ref{lem6}. Suppose $r$ is a number small enough such that
\begin{equation}\label{eq29}
a^{n_r^2}\leq r<a^{(n_r+1)^2}
\end{equation}
for some $n_r\in\mathbb{N}$. Let $Y_r=Y_{n_r}=[0,a^{n_r^2}]$, then we have
\begin{center}
$\nu(Y_{n_r})=(1-a)(a^{n_r}+a^{n_r+1}+\cdots)=(1-a)\cfrac{a^{n_r}}{(1-a)}=a^{n_r}$.
\end{center}
For any $x\in Y_{n_r}$, since $Y_{n_r}\subset B(x,r)$, we have
\begin{center}
$\nu(B(x,r))\geq \nu(Y_{n_r})=a^{n_r}$.
\end{center}
So 
\begin{center}
$\nu(B(x,r))\nu(Y_r)\geq a^{2n_r}$
\end{center}
for any $x\in Y_r$. Considering (\ref{eq29}), the condition (\ref{eq28}) is satisfied in this case. So by 
Lemma \ref{lem6}, we have
\begin{center}
$dim_C(\nu)=dim_{MC}(\nu)=0$.
\end{center}
\end{proof}

Obviously for the sequence of measures $\{\nu_n\}_{n\in\mathbb{N}}$ in Example \ref{exm7}, we have
\begin{center}
$dim_C(\nu_n)=dim_{MC}(\nu_n)=1$.
\end{center}
Combining this with Lemma \ref{lem7} and Proposition \ref{pro1}, we conclude that the measure dimension mappings $dim_C$ and $dim_{MC}$ can not be upper semi-continuous under the TV topology in general.

We end the section by some results on comparing the dimensions between two comparable measures.
\begin{lemma}\label{lem3}
For any two measures $\nu_1, \nu_2\in\mathcal{M}(X)$, if $\nu_1$ is absolutely continuous with respect to $\nu_2$, then the induced upper measure-dimension mapping satisfies
\begin{center}
$\dim^U(\nu_1)\leq \dim^U(\nu_2)$.
\end{center}
\end{lemma}
\begin{proof}
For $A\in\mathcal{B}$,  if $\nu_2(X\setminus A)=0$, since $\nu_1$ is absolutely continuous with respect to $\nu_2$, we have $\nu_1(X\setminus A)=0$, so
\begin{center}
$dim^U(\nu_2)=\inf\{dim(A): \nu_2(X\setminus A)=0\}\geq\inf\{dim(A): \nu_1(X\setminus A)=0\}=dim^U(\nu_1)$.
\end{center}
\end{proof}
One is recommended to compare the lemma with \cite[p220(a)]{MMR}. These conclusions again show duality of the two measure-dimension mappings.
\begin{corollary}\label{cor2}
For any two measures $\nu_1, \nu_2\in\mathcal{M}(X)$, if $\nu_1$ is equivalent to  $\nu_2$, then 
\begin{center}
$\dim^U(\nu_1)= \dim^U(\nu_2)$
\end{center}
for any induced upper measure-dimension mapping.
\end{corollary}

See also \cite[Lemma 3.1(3)]{HS}.

\section{Applications of the semi-continuity results to dimensions of measures originated from dynamical systems}\label{sec5}

Most of our results until now are set on measures on ambient space $X$ without dynamics on it. In this section we focus on probability measures carrying some dynamical structures on $X$. 

Let $T: X\rightarrow X$ be a transformation. Denote by $\mathcal{M}_\sigma(X)$ and $\mathcal{M}_e(X)$ respectively to be the collections of all the invariant and ergodic probability measures on $(X,\mathcal{B})$, with respect to $T$. The invariant and ergodic probability measures are crucial in many dynamical systems, while dimensions of these measures usually reflect important properties on these dynamical systems. A transformation $T: X\rightarrow X$ is said to be \emph{inverse-dimension-expanding} with respect to some dimensional mapping $dim$ if there exists some $A\in\mathcal{B}$, such that 
\begin{center}
$dim(T^{-1}(A))>dim(A)$.
\end{center}

\begin{Young's Lemma}
For a transformation $T$ on $(X,\mathcal{B})$ such that $T$ is not inverse dimension-expanding with respect to $dim$, then we have
\begin{center}
$dim^L(\nu)= dim^U(\nu)$
\end{center}
for any ergodic measure $\nu\in \mathcal{M}_e(X)$ with respect to $T$.
\end{Young's Lemma} 
\begin{proof}
See \cite[P115]{You}.
\end{proof}

Combining Young's Lemma and Corollary \ref{cor2}, we have the following result.

\begin{corollary}
For two equivalent measures $\nu_1, \nu_2\in\mathcal{M}(X)$ with one of them being ergodic with respect to some non-inverse-dimension-expanding transformation $T$ on $X$, we have
\begin{center}
$\dim^L(\nu_1)=\dim^U(\nu_1)=\dim^L(\nu_2)=\dim^U(\nu_2)$.
\end{center}
\end{corollary}

This corollary suggests a possible way to deal with the dimensions $\dim^L(\nu)$ and $\dim^U(\nu)$ of some (non-invariant) measure $\nu\in\mathcal{M}(X)$ simultaneously. That is, for some concerning measure $\nu\in\mathcal{M}(X)$, if one can find an ergodic measure $\nu_e\in\mathcal{M}_e(X)$ equivalent with $\nu$, then the dimension of  $\nu_e$ gives the dimension of $\nu$. As there are dynamical structure with respect to $\nu_e$, the dimension of  $\nu_e$ may be easier to be calculated. In case of non-existence of such an ergodic measure equivalent with $\nu$, one may resort to the ergodic decomposition of  $\nu$. The method still has some meaning for an invariant measure  $\nu$, though an invariant measure $\nu$ absolutely continuous with respect to an ergodic one $\nu_e$ necessarily satisfies
\begin{center}
$\nu=\nu_e$
\end{center}
according to \cite[P153, Remarks(1)]{Wal}.

In some cases it is difficult to decide the dimensions of ergodic (invariant) measures admitting some singularity (for example, when the measure has infinite entropy with respect to the transformation $T$). In these cases our semi-continuity of measures-dimension mappings under the setwise topology together with Young's Lemma provide a systematic way to approximate the dimensions of these singular ergodic (invariant) measures by sequences of dimensions of non-singular ergodic measures converging setwisely to the singular ones. The method usually involves the following steps.

\begin{enumerate}[1.]

\item For ones' target (singular or non-singular) measure $\nu$ on $X$ with a transformation $T$ being not inverse-dimension-expanding, find a sequence of non-singular measures $\{\nu_n\}_{n\in\mathbb{N}}$ converging to the target measure $\nu$ under the setwise topology. 

\item Show the existence of ergodic measures $\{\nu_n^*\}_{n\in\mathbb{N}}$ (with respect to some appropriate dynamical structures on $X$) equivalent with the ones $\{\nu_n\}_{n\in\mathbb{N}}$ respectively for any $n\in\mathbb{N}$ (sometimes they coincide with each other).

\item Apply Young's Lemma and our semi-continuity results-Theorem \ref{thm1} and \ref{thm4} to show the convergence of the sequence of dimensions of the  measures $\{\nu_n\}_{n\in\mathbb{N}}$ and $\{\nu_n^*\}_{n\in\mathbb{N}}$ to the dimension of the concerning measure $\nu$.
\end{enumerate}

For example, as a tool to tackle the obstacle of exploding entropy of some measure with respect to some partitions on the $X$ space, we have the following result, as an extension of \cite[Main Theorem]{You}.
\begin{corollary}\label{cor8}
Let $X$ be a (non-compact) two Riemannian manifold, $T: X\rightarrow X$ is a $C^2$ endomorphism. For an ergodic measure $\nu\in\mathcal{M}(X)$ with respect to $T$, if one can find a sequence of ergodic measures $\{\nu_n\in\mathcal{M}(X)\}_{n\in\mathbb{N}}$  with respect to a corresponding sequence of $C^2$ endomorphisms $\{T_n: X\rightarrow X\}_{n\in\mathbb{N}}$  such that
\begin{center}
$\nu_n\stackrel{s}{\rightarrow}\nu$ 
\end{center}
as $n\rightarrow\infty$ and $0\leq h_{\nu_n}(T_n), \lambda_1(T_n), \lambda_2(T_n)<\infty$ 
for any $n\in\mathbb{N}$, then
\begin{center}
$
\begin{array}{ll}
& dim_H^L(\nu)=dim_H^U(\nu)=\lim_{n\rightarrow\infty} dim_H^L(\nu_n)=\lim_{n\rightarrow\infty}  dim_H^U( \nu_n)\\
=&\lim_{n\rightarrow\infty} h_{\nu_n}(T_n)\Big(\cfrac{1}{\lambda_1(T_n)}-\cfrac{1}{\lambda_2(T_n)}\Big),
\end{array}
$
\end{center}
in which $\lambda_1(T_n)\geq \lambda_2(T_n)$ are the Lyapunov exponents of the maps $T_n$ and $T_n^{-1}$ respectively (they are constants for \emph{a.e.} $x\in X$ due to ergodicity of $T_n$).
\end{corollary} 
\begin{proof}
First, considering Young's Lemma, apply \cite[Main Theorem]{You} to the measure $\nu_n$ with finite entropy and Lyapunov exponents, we have
\begin{equation}\label{eq9}
dim_H^L(\nu_n)=dim_H^U(\nu_n)=h_{\nu_n}(T_n)\Big(\cfrac{1}{\lambda_1(T_n)}-\cfrac{1}{\lambda_2(T_n)}\Big).
\end{equation}
Since $\nu_n\stackrel{s}{\rightarrow}\nu$ as $n\rightarrow\infty$,  apply Theorem \ref{thm1} to the sequence of measures, we get
\begin{equation}\label{eq6}
\liminf_{n\rightarrow\infty} dim^U (\nu_n)\geq dim^U (\nu),
\end{equation}
while applying \ref{thm4} to the sequence of measures, we get
\begin{equation}\label{eq7}
\limsup_{n\rightarrow\infty} dim^L (\nu_n)\leq dim^L (\nu).
\end{equation}
Due to the ergodicity and Young's Lemma we have $dim^U (\nu)=dim^L (\nu)$, combing with (\ref{eq6}) and (\ref{eq7}) together we get
\begin{equation}\label{eq8}
\limsup_{n\rightarrow\infty} dim^L (\nu_n)\leq dim^L (\nu)=dim^U (\nu)\leq \liminf_{n\rightarrow\infty} dim^U (\nu_n).
\end{equation}
(\ref{eq9}) and (\ref{eq8}) together justify the corollary.
\end{proof}



\begin{thebibliography}{88}

\bibitem[Bil1]{Bil1} P. Billingsley, Convergence of Probability Measures, 2nd Edition, John Wiley \& Sons, 1999.

\bibitem[Bil2]{Bil2}  P. Billingsley, Probability and measure, 3rd Edition, John Wiley \& Sons, 1995.

\bibitem[CKP]{CKP} J. T. Cox, A. Klenke and E. A. Perkins, Convergence to equilibrium and linear
systems duality, In Stochastic Models (Ottawa, ON, 1998), CMS Conference Proceedings 26, 41-66. Amer. Math. Soc., Providence, RI. 

\bibitem[Cut]{Cut} C. D. Cutler, Some results on the behavior and estimation of the fractal dimensions of distributions on attractors, Journal of Statistical Physics, volume 62, pages 651-708 (1991).

\bibitem[Doo]{Doo} J. Doob, Measure Theory, Graduate Texts in Mathematics, 143, Springer-Verlag New York, 1994.

\bibitem[Fal1]{Fal1} K. Falconer, Fractal Geometry, Mathematical Foundations and Applications, 3rd Edition,  John Wiley \& Sons, 2014.

\bibitem[Fal2]{Fal2} K. Falconer, Techniques in Fractal Geometry, John Wiley \& Sons Ltd., Chichester,
1997.


\bibitem[FKZ1]{FKZ1} E. Feinberg, P. Kasyanov and M. Zgurovsky, Convergence of probability measures and Markov decision models with incomplete information, Proceedings of the Steklov Institute of Mathematics, December 2014, Volume 287, Issue 1, pp 96-117.

\bibitem[FKZ2]{FKZ2} E. Feinberg, P. Kasyanov and M. Zgurovsky, Partially observable total-cost Markov decision processes with weakly continuous transition probabilities, Mathematics of Operations Research, Volume 41, Issue 2, Pages 377-744, May 2016.

\bibitem[GP1]{GP1} P. Grassberger and I. Procaccia, Characterization of Strange Attractors, Physical Review Letters. 50 (5)(1983): 346-349. 

\bibitem[GP2]{GP2} P. Grassberger and I. Procaccia, Measuring the Strangeness of Strange Attractors, Physica D: Nonlinear Phenomena. 9(1-2)(1983): 189-208.

\bibitem[GR]{GR} J. K. Ghosh and R. V. Ramamoorthi, Bayesian Nonparametrics, Springer, New York, 2003.

\bibitem[HL]{HL} O. Hernandez-Lerma and J. Lasserre,  Markov Chains and Invariant Probabilities, Progress in Mathematics, Birkh\"auser Basel, 2003.

\bibitem[HS]{HS} M. Hochman and P. Shmerkin, Local entropy averages and projections of fractal measures, Ann. of Math. (2) 175 1001-1059, 2012. 

\bibitem[Kal1]{Kal1} O. Kallenberg, Random Measures, Theory and Applications. In: Probability Theory and Stochastic Modelling, vol. 77, Springer-Verlag, New York, 2017.

\bibitem[Kal2]{Kal2} O. Kallenberg, Foundations of Modern Probability, 2nd ed. New York: Springer, 2002.

\bibitem[Kle]{Kle} A. Klenke, Probability Theory: A Comprehensive Course, 2rd Edition, Springer, 2013.

\bibitem[Las]{Las} J. Lasserre, On the setwise convergence of sequences of measures, Journal of Applied Mathematics and Stochastic Analysis, 10:2 (1997), 131-136.

\bibitem[LY]{LY} T. Linder and S. Y\"uksel, Optimization and convergence of observation channels in stochastic control, SIAM J. Control Optim. Vol. 50, No. 2, pp. 864-887.

\bibitem[Mat]{Mat} P. Mattila, Geometry of sets and measures in Euclidean spaces, Fractals and rectifiability, Cambridge Studies in Advanced Mathematics, 44. Cambridge University Press, Cambridge, England, 1995.

\bibitem[MMR]{MMR}P. Mattila, M. Mor\'an and J. M. Rey, Dimension of a measure, Studia Math. 142 (2000), no. 3, 219-233.

\bibitem[MR]{MR} M. Mor\'an and J. M. Rey, Singularity of Self-Similar Measures with Respect to Hausdorff Measures, Transactions of the American Mathematical Society, Vol. 350, No. 6 (Jun., 1998), pp. 2297-2310.

\bibitem[Pes]{Pes} Y. Pesin, On rigorous mathematical definitions of correlation dimension and generalized spectrum for dimensions, Journal of Statistical Physics volume 71, pages 529-547(1993).

\bibitem[PS]{PS} K. Parthasarathy and T. Steerneman, A Tool in Establishing Total Variation Convergence, Proceedings of the American Mathematical Society Vol. 95, No. 4 (Dec., 1985), pp. 626-630.

\bibitem[Wal]{Wal} P. Walters,  An introduction to ergodic theory, Graduate Texts in Mathematics, 79. Springer, New York-Berlin, 1982.

\bibitem[You]{You} L. Young, Dimension, entropy and Lyapunov exponents, Ergodic Theory Dynam. Systems 2(1982), no. 1, 109-124.


\end{thebibliography}
\end{document}